\documentclass[reqno, 12pt]{amsart}
\usepackage{amssymb,amsmath,amsfonts,amsthm,comment,mathrsfs,times,graphicx}
\usepackage[bookmarksnumbered, plainpages]{hyperref}
\usepackage{color}
\usepackage[english]{babel}
\usepackage[all,cmtip]{xy}
\usepackage{lmodern}
\usepackage{enumitem}
\usepackage{geometry}
\newcounter{dummy} \numberwithin{dummy}{section}
\newtheorem{theo}[dummy]{Theorem}
\newtheorem{coro}[dummy]{Corollary}

\newtheorem{pro}[dummy]{Proposition}
\newtheorem{deft}[dummy]{Definition}
\newtheorem{exe}[dummy]{Example}
\newtheorem{rem}[dummy]{Remark}

\newtheorem*{example*}{Exemple}
\newtheorem*{examples*}{Exemples}
\newtheorem*{thm*}{Theorem}
\newtheorem*{definition*}{D\'{e}finition}
\newtheorem*{lem*}{Lemme}
\newtheorem*{prop*}{proposition}
\newcommand{\F}{{\bf F}}
\newcommand{\Z}{{\bf Z}}
\newcommand{\Q}{{\bf Q}}
\newcommand{\R}{{\bf R}}

\title[ A genus formula for the positive  \'{e}tale wild kernel]{ A genus formula for the  positive  \'{e}tale wild kernel}
\author[  Hassan Asensouyis, Jilali Assim \& Youness Mazigh]{Hassan Asensouyis$^{(*)}$, Jilali Assim$^{(\P)}$\;\&\; Youness Mazigh$^{(\P)}$}
\address{$^{(*)}$  Ibn Zohr University, Faculty of Applied Sciences, Ait Melloul,  80000 Agadir, Morocco.}
\address{$^{(\P)}$ Moulay Ismail University of Meknes, Faculty of Sciences, Department of Mathematics, B.P. 11201 Zitoune, 50000 Meknes, Morocco.}
\email{\textcolor[rgb]{0.00,0.00,1.00}{h.asensouyis@uiz.ac.ma }}
\email{\textcolor[rgb]{0.00,0.00,1.00}{j.assim@fs.umi.ac.ma}}
\email{\textcolor[rgb]{0.00,0.00,1.00}{y.mazigh@edu.umi.ac.ma}}
\begin{document}
	\maketitle
	\renewcommand{\abstractname}{Abstract}
	\begin{abstract}
		Let $F$ be a number field and let $i\geq 2$ be an  integer. In this paper, we  study the positive \'{e}tale wild kernel $\mathrm{WK}^{\mbox{\'{e}t},+}_{2i-2}F$, which is the twisted analogue of the  $2$-primary part of the narrow class group. If $E/F$ is a Galois extension of number fields with Galois group
		$G$, we prove  a genus formula relating the order of the groups $ (\mathrm{WK}^{\mbox{\'{e}t},+}_{2i-2}E)_{G}$ and  $\mathrm{WK}^{\mbox{\'{e}t},+}_{2i-2}F$.
	\end{abstract}
	\section{Introduction}
	Let $F$ be a number field and let $p$ be a prime number. For  a finite set $S$ of primes of $F$ containing the $p$-adic primes and the infinite primes, let $G_{F,S}$ be the Galois group of the maximal algebraic extension
	$F_{S}$ of $F$ which is unramified outside $S$. It is well known that for all integer $i\geq 2$, the kernel of the localization map
	\begin{equation*}
	\xymatrix@=2pc{
		H^{2}(G_{F,S},\Z_{p}(i))\ar[r]& \bigoplus_{v\in
			S}H^{2}(F_{v},\Z_{p}(i))}
	\end{equation*}
	is independent of the choice of the set $S$ \cite[\S 6, Lemma 1]{Sc 79} \cite[page 336]{Ko 03}. This kernel is called the \'{e}tale wild kernel \cite{Ng 92,Ko 93}, and denoted by $WK_{2i-2}^{\mbox{\'{e}t}}F$. It is the analogs of the $p$-part of the classical wild kernel $WK_{2}F$ which occurs in Moore's exact sequence of power norm symbols (cf. \cite{Milnor}).\\
	Let $E/F$ be a Galois extension with Galois group $G$. For a fixed odd  prime $p$, several authors have studied  Galois co-descent and proved  genus formulas  \cite{Ko-Mo,Griffiths,Assim 12,Hassan}, for the \'{e}tale wild kernel, which is analogue to the Chevalley genus formula for the class groups. In this paper we settle the case $p=2$. For this purpose,  we use a slight variant of cohomology, the so-colled totally positive Galois cohomology \cite[\S 5]{Ka 93}. More precisely, for all integer $i$, we are interested in the kernel of the map
	\begin{equation*}
	\xymatrix@=2pc{ H^{2}_{+}(G_{F,S},\Z_{2}(i))\ar[r]&
		\displaystyle{\bigoplus_{v\in
				S_{f}}}H^{2}(F_{v},\Z_{2}(i)).}
	\end{equation*}
	Here $S_{f}$ denotes the set of finite primes in $S$ and
	$H^{j}_{+}(.,.)$ denotes  the $j$-th totally positive Galois
	cohomology groups (Section \ref{Positive  wild kernel}). When $i=1$, this kernel is isomorphic to the  $2$-primary part of the narrow $S$-class group of $F$. For  $i\geq2$, we
	show that this kernel is independent of the set $S$; it is referred to as the positive \'{e}tale wild kernel, and is denoted by  $WK_{2i-2}^{\mbox{\'{e}t},+}F$. It is  analogue to the narrow $S$-class group of $F$, and fits into an exact sequence (Proposition \ref{Proposition DF and WF})
	\begin{equation*}
	\xymatrix@=1.5pc{ 0\ar[r]&D_{F}^{+(i)}/F^{\bullet^{2}}\ar[r]&
		D_{F}^{(i)}/F^{\bullet^{2}}\ar[r]&\oplus_{v\mid
			\infty}H^{1}(F_{v},\Z_{2}(i))\ar[r]&
		WK^{\mbox{\'{e}t},+}_{2i-2}F\ar[r]&	WK^{\mbox{\'{e}t}}_{2i-2}F\ar[r]&0}
	\end{equation*}
	where  $D_{F}^{(i)}$ (see \cite[Defintion 2.3]{Ko 03})  is the \'{e}tale Tate kernel and $D_{F}^{+(i)}/F^{\bullet^{2}}$  is the kernel  of the signature map
	\begin{equation*}
	\xymatrix@=2pc{\mathrm{sgn}_{F}:\; H^{1}(F,\Z_{2}(i))/2\cong D_{F}^{(i)}/F^{\bullet^{2}} \ar[r]&(\Z/2)^{r_{1}}.}
	\end{equation*}
	Here $r_{1}$ denotes  the number of real places of $F$.
	In particular,
	\begin{itemize}
		\item for  $i$ even, we have  \begin{equation*}
		WK_{2i-2}^{\mbox{\'{e}t},+}F\cong WK_{2i-2}^{\mbox{\'{e}t}}F.
		\end{equation*}
		\item for $i$ odd, we have an exact sequence
		\begin{equation*}
		\xymatrix@=1.5pc{0\ar[r]& (\Z/2)^{\delta_{i}(F)}\ar[r]& 	WK^{\mbox{\'{e}t},+}_{2i-2}F\ar[r]&WK^{\mbox{\'{e}t}}_{2i-2}F\ar[r]&0,}
		\end{equation*}
		where $\delta_{i}(F)$ is the $2$-rank of the cokernel of the signature map $\mathrm{sgn}_{F}$.
	\end{itemize}
	\vskip 6pt
	For a place $v$ of $F$, let $G_{v}$ denote the decomposition group of $v$ in $E/F$. Define the plus normic subgroup  $H^{1,\mathcal{N}}_{+}(F,\Z_{2}(i))$ to be	the kernel  of the map
	\begin{equation*}
	\xymatrix@=1.5pc{  H^{1}_{+}(F,\Z_{2}(i))\ar[r]&
		\oplus_{v\in S_{f}}
		\frac{H^{1}(F_{v},\Z_{2}(i))}{N_{G_{v}}H^{1}(E_{w},\Z_{2}(i))}}
	\end{equation*}
	where $N_{G_{v}}=\sum_{\sigma\in G_{v}}\sigma$ is the norm map, and if $v$ is a prime of $F$, we denote by $w$ a prime of $E$ above $v$.\vskip 6pt   We prove the following genus formula for the positive  \'{e}tale wild kernel.
	\begin{thm*}\label{main theo}
		Let $E/F$ be a Galois extension of number fields  with
		Galois group $G$. Then for every $i\geq 2$, we have
		\begin{equation*}
		\frac{|(WK_{2i-2}^{\mbox{\'{e}t},+}E)_{G}|}{|WK_{2i-2}^{\mbox{\'{e}t},+}F|}=\frac{|X^{(i)}_{E/F}|.\prod_{v\in
				S_{f}}|H_{1}(G_{v},H^{2}(E_{w},\Z_{2}(i)))|}
		{|H_{1}(G,H^{0}(E,\Q_{2}/\Z_{2}(1-i))^{\vee})|.[H^{1}_{+}(F,\Z_{2}(i)):
			\mathcal{H}^{1,\mathcal{N}}_{+}(F,\Z_{2}(i))]}
		\end{equation*}
	\end{thm*}
The group $X^{(i)}_{E/F}$ (Definition $\ref{Xi}$) has order at most $|H_{2}(G,H^{0}(E,\Q_{2}/\Z_{2}(1-i))^{\vee})|$, and is trivial  if the
canonical morphism
\begin{equation*}
\xymatrix@=1.5pc{\kappa:\;	H_{2}(G,\oplus_{w\in S_{f}}H^{2}(E_{w},\Z_{2}(i)))\ar[r]& 	H_{2}(G,H^{0}(E,\Q_{2}/\Z_{2}(1-i))^{\vee})}
\end{equation*}
is surjective.\vskip 6pt
In particular, if $E/F$ is a relative quadratic extension of number fields,
the  order of the group $X^{(i)}_{E/F}$ is at most $2$. In this case we give, in the last section, a genus formula involving the positive Tate kernel $D^{+(i)}_{F}$. Roughly speaking, let $F_{\infty}$ (resp. $F_{v,\infty}$) denote the cyclotomic $\Z_{2}$-extension of $F$ (resp. $F_{v}$) and   let $R_{E/F}$ be the set of both  finite  primes tamely ramified in $E/F$ and $2$-adic primes such that $E_{w}\cap F_{v,\infty}\neq E_{w}$. Then for any odd  integer $i\geq 2$, we have
		\begin{enumerate}[label=(\roman*)]
			\item if $E\subseteq F_{\infty}$, the positive \'{e}tale wild kernel satisfies Galois codescent;
			\item if $E\nsubseteq F_{\infty}$,
			\begin{equation*}
			\frac{|(WK_{2i-2}^{\mbox{\'{e}t},+}E)^{G}|}{|WK_{2i-2}^{\mbox{\'{e}t},+}F|}=\frac{2^{r(E/F)-1+t}}{
				[D_{F}^{+(i)}:D_{F}^{+(i)}\cap N_{G}E^{\bullet}]}
			\end{equation*}
			where $r(E/F)=|R_{E/F}|$ and $t\in\{0,1\}$. Moreover, $t=0$ if $R_{E/F}\neq\emptyset$.   	
		\end{enumerate}
	\section{Positive  \'{e}tale wild kernel}
	\subsection{Totally positive Galois cohomology}\label{Positive  wild kernel} Let $F$ be a number field and  let $S$ be a finite set of primes of $F$
	containing the set $S_{2}$ of dyadic primes  and the set $S_{\infty}$ of archimedean primes. For a place $v$ of $F$, we denote by $F_{v}$ the completion of $F$ at $v$,
	and by $G_{F_{v}}$ the absolute Galois group of $F_{v}$.\vskip 6pt  For a discrete or a compact  $\Z_{2}[[\mathrm{G_{F,S}}]]$-module $M$, we write $M_{+}$
	for the cokernel of the map
	\begin{equation*}
	 \xymatrix@=2pc{ M\ar[r]&
		\displaystyle{\oplus_{v\mid\infty}\mathrm{Ind}^{G_{F}}_{G_{F_{v}}}}M,}
	\end{equation*}
	where
	$\mathrm{Ind}^{G_{F}}_{G_{F_{v}}}M$ denotes the induced module.
	Hence we have the  exact sequence
	\begin{equation*}
	\xymatrix@=2pc{0\ar[r]&M\ar[r]&\displaystyle{\oplus_{v\mid
				\infty}\mathrm{Ind}^{G_{F}}_{G_{F_{v}}}}M\ar[r]&M_{+}\ar[r]&0.}
	\end{equation*}
	Following \cite[\S 5]{Ka 93}, we define the $n$-$\mathrm{th}$ totally
	positive Galois cohomology group $H^{n}_{+}(G_{F,S},M)$ of $M$ by
	\begin{equation*}
	H^{n}_{+}(G_{F,S},M):=H^{n-1}(G_{F,S},M_{+}).
	\end{equation*}
	Recall  from  \cite[\S 5]{Ka 93}  some facts about the totally positive Galois cohomology.
	\begin{pro}\label{Proposition properties of positive cohomology}
		We have the following properties:
		\begin{enumerate}[label=(\roman*)]
			\item There is a long exact sequence
			\begin{equation*}
			\xymatrix@=1pc{\cdots\ar[r]&H^{n}_{+}(G_{F,S},M)\ar[r]&
				H^{n}(G_{F,S},M)\ar[r]&
				\displaystyle{\oplus_{v\mid\infty}}H^{n}(F_{v},M)\ar[r]&H^{n+1}_{+}(G_{F,S},M)\ar[r]&\cdots}
			\end{equation*}
			\item $H^{n}_{+}(G_{F,S},M)=0$ for all $n\neq1,2$.
			\item If $E/F$ is an extension unramified outside $S$ with Galois group
			$G$ then there is a cohomological spectral sequence
			\begin{equation*}
			H^{p}(G,H^{q}_{+}(G_{E,S},M))\Longrightarrow
			H^{p+q}_{+}(G_{F,S},M).
			\end{equation*}
		\end{enumerate}
	\hfill$\square$
	\end{pro}
We also have a Tate spectral sequence
	\begin{equation*}
H_{p}(G,H^{q}_{+}(G_{E,S},M))\Longrightarrow
H^{p+q}_{+}(G_{F,S},M).
\end{equation*}

  Hence,  for a finite $2$-primary Galois module $M$, we have an isomorphism
	
	\begin{equation*}
	\xymatrix@=2pc{ H^{2}_{+}(G_{E,S},M)_{G}\ar[r]^-{\sim}&
		H^{2}_{+}(G_{F,S},M)}
	\end{equation*}
	(\cite[Lemma 6.4]{Weible}). In particular, by passing to the inverse limit, 	the corestriction  map induces an isomorphism
	\begin{equation}\label{cor iso}
		\xymatrix@=2pc{ H^{2}_{+}(G_{E,S},\Z_{2}(i))_{G}\ar[r]^-{\sim}&
		H^{2}_{+}(G_{F,S},\Z_{2}(i)).}
	\end{equation}
	Recall the local duality Theorem (e.g.\,\cite[Corollary
	I.2.3]{Milne}): For $n=0,1,2$ and for every place $v$ of $F$, the cup product
	\begin{equation}\label{local duality}
	\begin{array}{cccccc}
	H^{n}(F_{v},M) & \times& H^{2-n}(F_{v},M^{\ast}) & \xymatrix@=1.5pc{
		\ar[r]&} &
	H^{2}(F_{v},\mu_{2^{\infty}})\simeq \Q_{2}/\Z_{2},&
	\mbox{if $v$ is finite}\\
	&   &    &   &    \\
	\widehat{H}^{n}(F_{v},M) & \times& \widehat{H}^{2-n}(F_{v},M^{\ast}) & \xymatrix@=1.5pc{ \ar[r]&} &
	H^{2}(F_{v},\mu_{2^{\infty}}),&
	\mbox{if $v$ is infinite}
	\end{array}
	\end{equation}
	is a perfect pairing, where $\widehat{H}^{n}(F_{v},.)$
	is the Tate cohomology group, $\mu_{2^{\infty}}$ is the group of all roots of unity of $2$-power order, and  $(.)^{\ast}$ means the Kummer
	dual: $M^{\ast}=\mathrm{Hom}(M,\mu_{2^{\infty}})$.\vskip 6pt
	We have an analogue of the Poitou-Tate long exact sequence
	\begin{pro}\label{exact sequence of finie places} Let $S_{f}$ denote the set of finite places in $S$. Then there is a
		long exact sequence
		\begin{equation*}
		\xymatrix@=1.5pc{\oplus_{v\in
				S_{f}}H^{0}(F_{v},M)\ar@{^{(}->}[r]&H^{2}(G_{F,S},M^{\ast})^{\vee}\ar[r]&H^{1}_{+}(G_{F,S},M)\ar[r]& \oplus_{v\in S_{f}}H^{1}(F_{v},M)\ar[d]\\
			H^{0}(G_{F,S},M^{\ast})^{\vee}& \oplus_{v\in
				S_{f}}H^{2}(F_{v},M)\ar@{->>}[l]&H^{2}_{+}(G_{F,S},M)\ar[l]&
			H^{1}(G_{F,S},M^{\ast})^{\vee}\ar[l]
		}
		\end{equation*}
		where the subscript $(.)^{\vee}$ refers to the
		Pontryagin dual: $M^{\vee}=\mathrm{Hom}(M,\Q_{2}/\Z_{2})$.
	\end{pro}
	\begin{proof}
		See \cite[Proposition 2.6]{Mazigh}.
	\end{proof}
	For a  $\Z_{2}[[G_{F,S}]]$-module $M$ and
	$n=1,2$ we define the groups  $\mathrm{III}_{S}^{n}(M)$   and
	$\mathrm{III}_{S}^{n,+}(M)$ to be the kernels of the localization maps
	\begin{equation*}
	\mathrm{III}_{S}^{n}(M):=\ker(\xymatrix@=1.5pc{H^{n}(G_{F,S},M)\ar[r]& \bigoplus_{v\in
			S_{f}}H^{n}(F_{v},M)})
	\end{equation*}
	and
	\begin{equation*}
	\mathrm{III}_{S}^{n,+}(M):=\ker(\xymatrix@=1.5pc{
		H^{n}_{+}(G_{F,S},M)\ar[r]& \bigoplus_{v\in S_{f}}H^{n}(F_{v},M)}).
	\end{equation*}
	We state a Poitou-Tate duality in the case $p=2$ as a consequence of Proposition
	\ref{exact sequence of finie places} and local duality (\ref{local duality}).
	\begin{coro}\label{ Global duality}
		Let $n=1,2$. Then there is a perfect pairing
		\begin{equation*}
		\begin{array}{ccccc}
		\mathrm{III}_{S}^{n,+}(M) & \times & \mathrm{III}_{S}^{3-n}(M^{\ast}) & \xymatrix@=1.5pc{ \ar[r]&} &
		\Q_{2}/\Z_{2}.
		\end{array}
		\end{equation*}
	\end{coro}
	\begin{proof} By Proposition \ref{exact sequence of finie
			places} we have the  exact sequences
		\begin{equation*}
		\xymatrix@=2pc{0\ar[r]&\oplus_{v\in
				S_{f}}H^{0}(F_{v},M)\ar[r]&H^{2}(G_{F,S},M^{\ast})^{\vee}\ar[r]&
			\mathrm{III}_{S}^{1,+}(M)\ar[r]&0}
		\end{equation*}
		and
		\begin{equation*}
		\xymatrix@=2pc{\oplus_{v\in S_{f}}H^{1}(F_{v},M)\ar[r]& H^{1}(G_{F,S},M^{\ast})^{\vee}\ar[r]&
			\mathrm{III}_{S}^{2,+}(M)\ar[r]&0.}
		\end{equation*}
		Dualizing  these exact sequences and using the local duality $(\ref{local duality})$, we get
		\begin{equation*}
		\mathrm{III}_{S}^{1,+}(M)^{\vee}\cong  \mathrm{III}_{S}^{2}(M^{\ast})\quad \mbox{and}\quad  \mathrm{III}_{S}^{2,+}(M)^{\vee}\cong  \mathrm{III}_{S}^{1}(M^{\ast}).
		\end{equation*}
	\end{proof}
	
	\subsection{Signature} In this subsection we recall  some properties of the signature map (see e.g. \cite{Ko 03}, \cite[\S 1.2]{Assim mova}). For any real place $v$ of the number field $F$, let $i_{v}: F\longrightarrow \R$ denote the
	corresponding real embedding. The natural signature maps
	$\mathrm{sgn}_{v}: F^{\bullet}\longrightarrow \Z/2\Z$
	(where $\mathrm{sgn}_{v}(x)=0$ or $1$ according to $i_{v}(x)>0$ or
	not) give rise to the following surjective map
	\begin{equation*}
	\begin{array}{cccc}
	&F^{\bullet}/F^{\bullet^{2}} & \xymatrix@=1.5pc{\ar[r]&} & \oplus_{v\; real}\Z/2\Z \\
	&x & \xymatrix@=1.5pc{ \ar@{|->}[r]&} & (\mathrm{sgn}_{v}(x))_{v\; real}
	\end{array}
	\end{equation*}
	The exact sequence of $G_{F}$-modules
	\begin{equation*}
	\xymatrix@=2pc{ 0\ar[r]&\Z_{2}(i)\ar[r]^-{2}& \Z_{2}(i)\ar[r]& \Z/2(i)\ar[r]&0}
	\end{equation*}
	gives rise to an exact sequence
	\begin{equation*}
	\xymatrix@=2pc{0\ar[r]& H^{1}(F,\Z_{2}(i))/2\ar[r]& H^{1}(F, \Z/2(i))\ar[r]& H^{2}(F,\Z_{2}(i)),}
	\end{equation*}
	where for an abelian group $A$, $A/2$ denotes the cokernel of the multiplication by $2$ on $A$.\vskip 6pt
	Since we have
	\begin{eqnarray*}
		H^{1}(F,\Z/2(i))&\cong& 	H^{1}(F,\Z/2(1))(i-1)\\
		&\cong& F^{\bullet}/F^{\bullet^{2}}(i-1)
	\end{eqnarray*}
	there exists a subgroup $D_{F}^{(i)}$ (the \'{e}tale Tate kernel) of $F^{\bullet}$ containing $F^{\bullet^{2}}$
	such that
	\begin{equation}\label{DF}
	D_{F}^{(i)}/F^{\bullet^{2}}\cong H^{1}(F,\Z_{2}(i))/2.
	\end{equation}
	We will consider the restriction of the above  signature map to the quotient $D_{F}^{(i)}/F^{\bullet^{2}}$:
	\begin{equation*}
	\xymatrix@=2pc{\mathrm{sgn}_{F}:\; D_{F}^{(i)}/F^{\bullet^{2}}\ar[r]& (\Z/2)^{r_{1}},}
	\end{equation*}
	where $r_{1}$ is the number of real places of $F$.\\
	
	Let $D_{F}^{+(i)}/F^{\bullet^{2}}$ be the kernel and $(\Z/2)^{\delta_{i}(F)}$ be the cokernel of $\mathrm{sgn}_{F}$, respectively. So we have an exact sequence
	\begin{equation*}
	\xymatrix@=1.5pc{ 0\ar[r]&D_{F}^{+(i)}/F^{\bullet\,2}\ar[r]&
		D_{F}^{(i)}/F^{\bullet\,2}\ar[r]^-{\mathrm{sgn}_{F}}&
		(\Z/2)^{r_{1}}\ar[r]& (\Z/2)^{\delta_{i}(F)}\ar[r]&0.}
	\end{equation*}
	If $i$ is an even integer, the
	signature map
	\begin{equation}\label{signature map tri}
	\xymatrix@=2pc{\mathrm{sgn}_{F}:\; D_{F}^{(i)}/F^{\bullet^{2}}\ar[r]& (\Z/2)^{r_{1}}}
	\end{equation}
	is trivial \cite[Proposition 1.2]{Assim mova}, and then $D_{F}^{+(i)}=D_{F}^{(i)}$.
	\subsection{ Positive \'{e}tale wild kernel} Following \cite{Ng 92,Ko 93}, the \'{e}tale wild kernel $WK_{2i-2}^{\mbox{\'{e}t}}F$ is the group
		\begin{equation*}
	WK_{2i-2}^{\mbox{\'{e}t}}F:=\ker( \xymatrix@=1.5pc{
		H^{2}(G_{F,S},\Z_{p}(i))\ar[r]& \bigoplus_{v\in
			S}H^{2}(F_{v},\Z_{p}(i))}).
	\end{equation*}
		For $i\geq2$, it is well known that the \'{e}tale wild kernel $WK_{2i-2}^{\mbox{\'{e}t}}F$ is independent of the set $S$ containing the $p$-adic primes and the infinite primes (\cite[\S 6, Lemma 1]{Sc 79}, \cite[page 336]{ Ko 03}).\vskip 6pt
	There have been much work on  the Galois co-descent for the \'{e}tale
	wild kernel at odd primes \cite{Ko-Mo,Griffiths,Assim 12,Hassan}. The case $p=2$ has been studied essentially in the classical case $i=2$ \cite{Ko-Mo,Ko-Mo 03,Griffiths}. The situation for $p=2$ is more complicated, since the cohomology groups $H^{k}(\R,\Z_{2}(i))$ and $H^{k}(\R,\Q_{2}/\Z_{2}(i))$  do not necessarily vanish, and the group  $H^{2}(G_{E,S},\Z_{2}(i))$ does not satisfy  Galois co-descent. This motivates the following definition of the positive \'{e}tale wild kernel.\vskip 6pt
	Let $S_{f}$ denote  the set of finite primes in $S$ and let $\mathcal{O}_{F,S}$ be the ring of $S$-integers of $F$. For all $i\in\Z$, recall the last three terms of the Poitou-Tate exact sequence (Proposition \ref{exact sequence of finie places}):
	\begin{equation*}
	\xymatrix@=2pc{	H^{2}_{+}(G_{F,S},\Z_{2}(i))\ar[r]&\oplus_{v\in S_{f}}H^{2}(F_{v},\Z_{2}(i))\ar[r]&H^{0}(G_{F,S},\Q_2/\Z_2(1-i))^{\vee}\ar[r]&0}
	\end{equation*}
	\begin{deft} Let $i\in\Z$. We define the  positive \'{e}tale wild kernel
		$WK_{2i-2}^{\mbox{\'{e}t},+}\mathcal{O}_{F,S}$ to be the kernel of the localization map
		\begin{equation*}
		\xymatrix@=2pc{ H^{2}_{+}(G_{F,S},\Z_{2}(i))\ar[r]&
			\oplus_{v\in S_{f}}H^{2}(F_{v},\Z_{2}(i)).}
		\end{equation*}
	\end{deft}
	\begin{rem}
		For $i=1$, the group $WK_{0}^{\mbox{\'{e}t},+}\mathcal{O}_{F,S}$ is isomorphic to the $2$-part of the narrow $S$-class group  $A_{F,S}^{+}$ of
		$F$. In particular it depends on the set $S$. Indeed, on the one hand
		Corollary $\ref{ Global duality}$
		shows that
		\begin{equation*}
		WK_{0}^{\mbox{\'{e}t},+}\mathcal{O}_{F,S}\cong \mathrm{III}_{S}^{1}(\Q_{2}/\Z_{2})^{\vee}.
		\end{equation*}
		On the other hand
		\begin{eqnarray*}
			\mathrm{III}_{S}^{1}(\Q_{2}/\Z_{2})&=&\ker(\xymatrix@=2pc{H^{1}
				(G_{F,S},\Q_{2}/\Z_{2})\ar[r]&\oplus_{v\in S_{f}}H^{1}(F_{v},\Q_{2}/\Z_{2}}))\\
			&=& \ker(\xymatrix@=2pc{\mathrm{Hom}(G_{F,S},\Q_{2}/\Z_{2})\ar[r]&
				\oplus_{v\in S_{f}} \mathrm{Hom}(G_{F_{v}},\Q_{2}/\Z_{2})})\\
			&\cong& (A_{F,S}^{+})^{\vee}.
		\end{eqnarray*}
		It follows that $WK_{0}^{\mbox{\'{e}t},+}\mathcal{O}_{F,S}\cong A_{F,S}^{+}$.
	\end{rem}
	Hence the  positive \'{e}tale wild kernel plays a similar role as the $2$-primary part of the narrow $S$-class group. We restrict our study to the case $i\geq 2$, and we will show that $WK_{2i-2}^{\mbox{\'{e}t},+}\mathcal{O}_{F,S}$ is independent of the set $S$ containing the $2$-adic primes and the infinite primes. However, all the results remain true for  $i\neq 1$ if we assume the finiteness of the Galois cohomology group $H^{2}(G_{F,S},\Z_{2}(i))$. Note that for $i=0$, this finiteness is equivalent to the Leopoldt conjecture, and for $i\geq 2$ this is true
	as a consequence of the finiteness of the $K$-theory groups  $K_{2i-2}\mathcal{O}_{F,S}$ and the connection between $K$-theory and \'{e}tale cohomology via Chern characters \cite{So 79, Dwyer}.\vskip 6pt
	The following proposition gives the link between the kernels  $	WK^{\mbox{\'{e}t},+}_{2i-2}\mathcal{O}_{F,S}$ and $WK^{\mbox{\'{e}t}}_{2i-2}F$.
	\begin{pro}\label{Proposition DF and WF}
		For all integer $i\geq 2$, there exists an exact sequence
		\begin{equation*}
		\xymatrix@=1.5pc{ 0\ar[r]&D_{F}^{+(i)}/F^{\bullet^{2}}\ar[r]&
			D_{F}^{(i)}/F^{\bullet^{2}}\ar[r]&\oplus_{v\mid
				\infty}H^{1}(F_{v},\Z_{2}(i))\ar[r]&
			WK^{\mbox{\'{e}t},+}_{2i-2}\mathcal{O}_{F,S}\ar[r]&	WK^{\mbox{\'{e}t}}_{2i-2}F\ar[r]&0.}
		\end{equation*}
		In particular,
		\begin{itemize}
			\item if $i$ is even, there is an isomorphism
			\begin{equation*}
			WK^{\mbox{\'{e}t},+}_{2i-2}\mathcal{O}_{F,S}\cong
			WK^{\mbox{\'{e}t}}_{2i-2}F,
			\end{equation*}	
			\item  if $i$ is odd, we have an exact sequence:
			\begin{equation*}
			\xymatrix@=1.5pc{ 0\ar[r]&
				(\Z/2)^{\delta_{i}(F)}\ar[r]&WK^{\mbox{\'{e}t},+}_{2i-2}\mathcal{O}_{F,S}\ar[r]&	WK^{\mbox{\'{e}t}}_{2i-2}F\ar[r]&0}
			\end{equation*}
			where  $\delta_{i}(F)$ is the $2$-rank of the cokernel of the signature map $\mathrm{sgn}_{F}$.
		\end{itemize}
	\end{pro}
	\begin{proof}
		On the one hand, the exact sequence
		\begin{equation*}
		\xymatrix@=2pc{0\ar[r]& \Z_{2}(i)\ar[r]^-{2}& \Z_{2}(i)\ar[r]& \Z/2(i)\ar[r]&0}
		\end{equation*}
		gives rise to an exact commutative diagram
		\begin{equation*}
		\xymatrix@=2pc{ 0\ar[r]& H^{1}(G_{F,S},\Z_{2}(i))/2\ar[r]\ar[d]& H^{1}(G_{F,S},\Z/2(i))\ar[d]\ar[r]&H^{2}(G_{F,S},\Z_{2}(i))\ar[d] \\
			0\ar[r]&\oplus_{v\mid\infty}H^{1}(F_{v},\Z_{2}(i))/2\ar[r] & \oplus_{v\mid \infty}H^{1}(F_{v},\Z/2(i))\ar[r]& \oplus_{v\mid \infty}H^{1}(F_{v},\Z_{2}(i))  }
		\end{equation*}		
		where the vertical maps are the localization maps. Since
		\begin{equation*}
		2H^{1}(F_{v},\Z_{2}(i))=0
		\end{equation*}
		for all infinite place $v$ of $F$, we get 	
		
		\begin{equation*}
		\xymatrix@=2pc{  H^{1}(G_{F,S},\Z_{2}(i))/2\ar@{^{(}->}[r]\ar[d]& H^{1}(G_{F,S},\Z/2(i))\ar[d]\\
			\oplus_{v\mid\infty}H^{1}(F_{v},\Z_{2}(i))\ar@{^{(}->}[r] & \oplus_{v\mid \infty}H^{1}(F_{v},\Z/2(i)) }
		\end{equation*}		
		Observe that the composite
		\begin{equation*}
		\xymatrix@=1.5pc{ D_{F}^{(i)}/F^{\bullet^{2}}\cong H^{1}(G_{F,S},\Z_{2}(i))/2\ar[r]&
			\oplus_{v\mid \infty}H^{1}(F_{v},\Z_{2}(i))\ar[r]&
			\oplus_{v\mid\infty}H^{1}(F_{v},\Z/2(i))\cong (\Z/2)^{r_{1}}}
		\end{equation*}		
		is the signature map
		\begin{equation*}
		\xymatrix@=2pc{ \mathrm{sgn}_{F}:\;H^{1}(F,\Z_{2}(i))/2\ar[r]&
			(\Z/2)^{r_{1}}}
		\end{equation*}
		and then
		\begin{equation}\label{D plus}
		D^{+(i)}_{F}/F^{\bullet^{2}}\cong \ker(\xymatrix@=2pc{H^{1}(G_{F,S},\Z_{2}(i))/2\ar[r]&
			\oplus_{v\mid \infty}H^{1}(F_{v},\Z_{2}(i))}).
		\end{equation}
		On the other hand, by the definition of totally positive Galois cohomology, we have the exact sequence
		\begin{equation*}
		\xymatrix@=1.5pc{H^{1}(G_{F,S},\Z_{2}(i))\ar[r]&\bigoplus_{v\mid
				\infty}H^{1}_{v}\ar[r]&H^{2}_{+}(G_{F,S},\Z_{2}(i))\ar[r]&
			H^{2}(G_{F,S},\Z_{2}(i))\ar[r]&\bigoplus_{v\mid\infty}H^{2}_{v}\ar[r]&0,}
		\end{equation*}
		where for $n=1\,\mbox{or}\,2$, $H^{n}_{v}$ denotes the cohomology group $H^{n}(F_{v},\Z_{2}(i))$.
		Since
		\begin{equation*}
		2H^{1}(F_{v},\Z_{2}(i))=0
		\end{equation*}
		for all infinite place $v$ of $F$, we get
		\begin{equation*}
		\xymatrix@=1.5pc{H^{1}(G_{F,S},\Z_{2}(i))/2\ar[r]&\bigoplus_{v\mid
				\infty}H^{1}_{v}\ar[r]&H^{2}_{+}(G_{F,S},\Z_{2}(i))\ar[r]&
			H^{2}(G_{F,S},\Z_{2}(i))\ar@{->>}[r]&\bigoplus_{v\mid\infty}H^{2}_{v}.}
		\end{equation*}
		Therefore, we have the following exact commutative diagram
		\begin{equation*}
		\xymatrix@=1.5pc{H^{1}(G_{F,S},\Z_{2}(i))/2\ar[r]&\bigoplus_{v\mid
				\infty}H^{1}_{v}\ar[r]&H^{2}_{+}(G_{F,S},\Z_{2}(i))\ar[r]\ar[d]&
			H^{2}(G_{F,S},\Z_{2}(i))\ar@{->>}[r]\ar[d]&\bigoplus_{v\mid\infty}H^{2}_{v}\ar@{=}[d]\\
			&   0\ar[r] & \bigoplus_{v\in S_{f}}H^{2}_{v}\ar[r]& \oplus_{v\in S}H^{2}_{v}\ar@{->>}[r]& \bigoplus_{v\mid\infty}H^{2}_{v} }
		\end{equation*}
		By the snake lemma and (\ref{D plus}), we obtain the exact sequence		
		\begin{equation*}
		\xymatrix@=1.5pc{ 0\ar[r]&D_{F}^{+(i)}/F^{\bullet^{2}}\ar[r]&
			D_{F}^{(i)}/F^{\bullet^{2}}\ar[r]&\oplus_{v\mid
				\infty}H^{1}(F_{v},\Z_{2}(i))\ar[r]&
			WK^{\mbox{\'{e}t},+}_{2i-2}\mathcal{O}_{F,S}\ar[r]&	WK^{\mbox{\'{e}t}}_{2i-2}F\ar[r]&0.}
		\end{equation*}		
		Since $\bigoplus_{v\mid\infty}H^{1}(F_{v},\Z_{2}(i))$ is isomorphic to $(\Z/2)^{r_{1}}$ if $i$ is odd, and is trivial if $i$ is even, we obtain
		
		\begin{itemize}
			\item for $i$  even, $	WK_{2i-2}^{\mbox{\'{e}t},+}\mathcal{O}_{F,S}\cong WK^{\mbox{\'{e}t}}_{2i-2}F$, and
			\item for $i$  odd, we have the exact sequence
			 	\begin{equation}\label{exat plus etale}
			\xymatrix@=2pc{0\ar[r]& (\Z/2)^{\delta_{i}(F)}\ar[r]&  WK_{2i-2}^{\mbox{\'{e}t},+}\mathcal{O}_{F,S}\ar[r]& 	WK^{\mbox{\'{e}t}}_{2i-2}F\ar[r]&0,}
			\end{equation}
			where  $\delta_{i}(F)$ is the $2$-rank of the cokernel of the signature map $\mathrm{sgn}_{F}$.
		\end{itemize}
	\end{proof}
	The following corollary shows that $WK_{2i-2}^{\mbox{\'{e}t},+}\mathcal{O}_{F,S}$ is in fact independent of the set $S$ containing $S_{2\infty}=S_{2}\cup S_{\infty}$ for  $i\geq 2$.
	\begin{coro}\label{coro independent}
		For  $i\geq 2$, the positive \'{e}tale wild kernel
		$WK_{2i-2}^{\mbox{\'{e}t},+}\mathcal{O}_{F,S}$ is independent of the set $S$ containing $S_{2\infty}$.
	\end{coro}
	\begin{proof} Since $\delta_{i}(F)$ and $WK^{\mbox{\'{e}t}}_{2i-2}F$ are independent of the set $S$ \cite[page 336 ]{Ko 03}, the exact sequence (\ref{exat plus etale}) shows that 	the order of $	WK^{\mbox{\'{e}t},+}_{2i-2}\mathcal{O}_{F,S}$ is also independent of $S$. Therefore it suffices to prove that there exists an injective map from $WK_{2i-2}^{\mbox{\'{e}t},+}\mathcal{O}_{F,S_{2\infty}}$ to $WK_{2i-2}^{\mbox{\'{e}t},+}\mathcal{O}_{F,S}$. But this follows from the exact commutative diagram
		\begin{equation*}
		\xymatrix@=1.5pc{ 0\ar[r]&WK_{2i-2}^{\mbox{\'{e}t},+}\mathcal{O}_{F,S_{2\infty}}\ar[d]\ar[r]& H^{2}_{+}(G_{F,S_{2\infty}},\Z_{2}(i))\ar[r]\ar@{^{(}->}[d]& \oplus_{v\mid 2}H^{2}(F_{v},\Z_{2}(i))\ar@{^{(}->}[d]\\
			0\ar[r]&  WK_{2i-2}^{\mbox{\'{e}t},+}\mathcal{O}_{F,S}\ar[r]&
			H^{2}_{+}(G_{F,S},\Z_{2}(i))\ar[r]& \oplus_{v\in S_{f}}H^{2}(F_{v},\Z_{2}(i))}
		\end{equation*}
		where  the middle vertical  map is the inflation map.
	\end{proof}
	From now on, we make the notation
	\begin{equation*}
	WK^{\mbox{\'{e}t},+}_{2i-2}\mathcal{O}_{F,S}:=	WK^{\mbox{\'{e}t},+}_{2i-2}F,\;\mbox{for $i\geq 2$}.
	\end{equation*}
	We finish  this subsection by giving a description of $WK_{2i-2}^{\mbox{\'{e}t},+}F$ as an Iwasawa module. \vskip 6pt
	Let  $X^{\prime}_{\infty}$ be the
	Galois group of the maximal unramified $p$-extension of the
	cyclotomic $\Z_{p}$-extension of $F$, which is completely
	decomposed at all primes above $p$. It is well known that (\cite[Lemma 1,\S 6]{Sc 79}) 	for any odd prime $p$
	\begin{equation}\label{description of WK}
	WK_{2i-2}^{\mbox{\'{e}t}}F\cong X^{\prime}_{\infty}(i-1)_{\mathrm{Gal}(F(\mu_{p^{\infty}})/F)}.
	\end{equation}
 In the next proposition, we prove an analogue result in the case $p=2$.
	\vskip 6pt
	Let $F_{\infty}$ be  the cyclotomic	$\Z_{2}$-extension of $F$ with Galois group  $\Gamma=\mathrm{Gal}(F_{\infty}/F)$, and let
	$X^{\prime,+}_{\infty}$ be the Galois group of the maximal $2$-extension of $F_{\infty}$, which is unramified at finite places and
	completely decomposed at all primes above $2$.
	\begin{pro}\label{Prop ind}
		Let $i\geq 2$ be an integer. If either $i$ is odd, or $i$ is even and $\sqrt{-1}\in F$, then
		\begin{equation*}
		WK_{2i-2}^{\mbox{\'{e}t},+}F\cong X^{\prime,+}_{\infty}(i-1)_{\Gamma}.
		\end{equation*}
		In particular, in both cases we recover that the group $WK_{2i-2}^{\mbox{\'{e}t},+}F$ is
		independent of the set $S$ containing $S_{2\infty}$.
	\end{pro}
	\begin{proof} First observe that, if $i$ is odd or $i$ is even and $\sqrt{-1}\in F$, then
		\begin{equation}\label{Tate lemma}
		H^{1}(\Gamma,\Q_{2}/\Z_{2}(1-i))=0.
		\end{equation}
		Indeed, in both cases, $\Q_{2}/\Z_{2}(1-i)$ is a $\Gamma$-module and using \cite[\S XIII.1, Proposition 1]{Se 68}, we see that
		\begin{equation*}
		H^{1}(\Gamma,\Q_{2}/\Z_{2}(1-i))\cong \big(\Q_{2}/\Z_{2}(1-i)\big)/(\gamma-1).\big(\Q_{2}/\Z_{2}(1-i)\big),
		\end{equation*}
		where $\gamma$ is a topological generator of $\Gamma$. Hence
		\begin{equation*}
		H^{1}(\Gamma,\Q_{2}/\Z_{2}(1-i))=0.
		\end{equation*}
		Now we consider  the following  exact commutative diagram
		\begin{equation*}
		\xymatrix@=1.5pc{
			H^{1}(\Gamma,\Q_{2}/\Z_{2}(j))\ar@{^{(}->}[r]\ar[d]&
			H^{1}(G_{F,S},\Q_{2}/\Z_{2}(j))\ar[r]\ar[d]&
			H^{1}(G_{F_{\infty},S},\Q_{2}/\Z_{2}(j))^{\Gamma}\ar[d]\ar[r]&0\\
			 \bigoplus_{v\in
				S_{f}}H^{1}(\Gamma_{v},\Q_{2}/\Z_{2}(j))\ar@{^{(}->}[r]&
			\bigoplus_{v\in
				S_{f}}H^{1}(F_{v},\Q_{2}/\Z_{2}(j))\ar[r]&\bigoplus_{v\in
				S_{f}}H^{1}(F_{v,\infty},\Q_{2}/\Z_{2}(j))^{\Gamma}\ar[r]&0}
		\end{equation*}
		where $j=1-i$,
		$H^{1}(F_{v,\infty},\Q_{2}/\Z_{2}(j))=\bigoplus_{w\mid
			v}H^{1}(F_{w,\infty},\Q_{2}/\Z_{2}(j))$, and
		$\Gamma_{v}$ denotes the decomposition group of $v$ in
		$F_{\infty}/F$.
		By (\ref{Tate lemma}), we have
		\begin{equation*}
		H^{1}(\Gamma,\Q_{2}/\Z_{2}(j))=0\,\,\mbox{and}\,\,H^{1}(\Gamma_{v},\Q_{2}/\Z_{2}(j))=0\;\;\mbox{for all $v\in S_{f}$},
		\end{equation*}
		and then
		\begin{eqnarray*}
			\mathrm{III}_{S}^{1}(\Q_{2}/\Z_{2}(j)) &=&\ker (\xymatrix@=1.5pc{H^{1}(G_{F_{\infty},S},\Q_{2}/\Z_{2}(j))^{\Gamma}
				\ar[r]& \bigoplus_{v\in
					S_{f}}H^{1}(F_{v,\infty},\Q_{2}/\Z_{2}(j))^{\Gamma}})  \\
			&=&\mathrm{Hom}(X^{\prime,+}_{\infty},
			\Q_{2}/\Z_{2})(-j)^{\Gamma}.
		\end{eqnarray*}
		Hence, using the duality
		\begin{equation*}
		WK_{2i-2}^{\mbox{\'{e}t},+}F\cong
		\mathrm{III}_{S}^{1}(\Q_{2}/\Z_{2}(1-i))^{\vee}
		\end{equation*}
		(Corollary \ref{ Global duality}),
		we obtain the isomorphism
		\begin{equation*}
		WK_{2i-2}^{\mbox{\'{e}t},+}F\cong X^{\prime,+}_{\infty}(i-1)_{\Gamma}.
		\end{equation*}
	\end{proof}
	Let  $F_{\infty}=\cup_{n}F_{n}$ be the cyclotomic $\Z_{2}$-extension of $F$ and for
	$n\geq 0$,  $G_{n}=\mathrm{Gal}(F_{n}/F)$. The above description of the positive \'{e}tale wild kernel, leads immediately  to the following corollary:
	\begin{coro}
		If either $i$ is odd, or $i$ is even and $\sqrt{-1}\in F$, then the positive \'{e}tale wild kernel satisfies Galois co-descent in the cyclotomic $\Z_{2}$-extension:
		\begin{equation*}
		(WK_{2i-2}^{\mbox{\'{e}t},+}F_{n})_{G_{n}}\cong 	WK_{2i-2}^{\mbox{\'{e}t},+}F.
		\end{equation*}
	\end{coro}
Compare to  \cite[Theorem 2.18]{Ko-Mo}, which deals with the case $i=2$ and  $\sqrt{-1} \in F$. If $p$ is odd, the Galois co-descent holds  in the cyclotomic tower  as a consequence  of Schneider's description of the \'{e}tale wild kernel.
	\section{Genus formula}
	Let $E/F$ be a  Galois extension of number fields with Galois group
	$G$. Let $S$ denote the set of infinite places, $2$-adic places and
	those which ramify in $E/F$. We denote also by $S$ the set of places
	of $E$ above places in $S$. In the sequel we assume that $i\geq 2$.
	\vskip 6pt
	By the definition of $WK_{2i-2}^{\mbox{\'{e}t},+}F$ and Proposition \ref{exact sequence of finie places}, we have the exact sequence
	\begin{equation*}
	\xymatrix@=2pc{0\ar[r]&
		WK_{2i-2}^{\mbox{\'{e}t},+}F\ar[r]&H^{2}_{+}(G_{F,S},\Z_{2}(i))\ar[r]&
		\widetilde{\oplus}_{v\in
			S_{f}}H^{2}(F_{v},\Z_{2}(i))\ar[r]&0,}
	\end{equation*}
	where $\widetilde{\oplus}_{v\in
		S_{f}}H^{2}(F_{v},\Z_{2}(i))$ denotes the kernel of the surjective map
	\begin{equation*}
	\xymatrix@=2pc{ \oplus_{v\in
			S_{f}}H^{2}(F_{v},\Z_{2}(i))\ar@{->>}[r]&
		H^{0}(G_{F,S},\Q_{2}/\Z_{2}(1-i))^{\vee}}.
	\end{equation*}
	Then the corestriction map induces the exact commutative diagram
	\begin{equation}\label{diagramme commutive}
	\xymatrix@=1.5pc{
		&(WK_{2i-2}^{\mbox{\'{e}t},+}E)_{G}\ar[d]^-{N_{i}}\ar[r]&
		H^{2}_{+}(G_{E,S},\Z_{2}(i))_{G}\ar[d]^-{\wr}\ar[r]&
		(\widetilde{\oplus}_{w\in S_{f}}H^{2}(E_{w},\Z_{2}(i)))_{G}\ar[d]^-{N^{\prime}_{i}}\ar[r]&0\\
		0\ar[r]&
		WK_{2i-2}^{\mbox{\'{e}t},+}F\ar[r]&H^{2}_{+}(G_{F,S},\Z_{2}(i))\ar[r]&
		\widetilde{\oplus}_{v\in
			S_{f}}H^{2}(F_{v},\Z_{2}(i))\ar[r]&0}
	\end{equation}
	where the middle vertical map is an isomorphism by $(\ref{cor iso})$.
	Using the snake lemma, we get
	\begin{itemize}
		\item $\mathrm{coker}N_{i}\cong \ker N_{i}^{\prime}$.
		\item $\ker N_{i}\cong \mathrm{coker}(\widetilde{\alpha})$, where $\widetilde{\alpha}$ is the homology map
		\begin{equation*}
		\xymatrix@=2pc{
		 \widetilde{\alpha}:
		H_{1}(G,H^{2}_{+}(G_{E,S},\Z_{2}(i)))\ar[r]&
		H_{1}(G,\widetilde{\oplus}_{w\in
			S_{f}}H^{2}(E_{w},\Z_{2}(i))).}
		\end{equation*}
	\end{itemize}
We first determine $\mathrm{coker}N_{i}$, and then we give  a criterion of the surjectivity  of the morphism $N_{i}$. For this, the  exact commutative diagram
	\begin{equation}\label{exact commutaive diagram widetilde}
	\xymatrix@=1.5pc{ (\widetilde{\bigoplus}_{w\in
			S_{f}}H^{2}(E_{w},\Z_{2}(i)))_{G}\ar[r]\ar[d]^-{N_{i}^{\prime}}&
		(\displaystyle{\bigoplus_{w\in
				S_{f}}}H^{2}(E_{w},\Z_{2}(i)))_{G}\ar@{->>}[r]\ar[d]^-{\wr}&
		(H^{0}(E,\mathbb{Q}_{2}/\Z_{2}(1-i))^{\vee})_{G}\ar[d]^-{\wr}\\
		\widetilde{\bigoplus}_{v\in
			S_{f}}H^{2}(F_{v},\Z_{2}(i))\ar@{^{(}->}[r]&\bigoplus_{v\in
			S_{f}}H^{2}(F_{v},\Z_{2}(i))\ar@{->>}[r]&H^{0}(F,\Q_{2}/\Z_{2}(1-i))^{\vee}}
	\end{equation}
	shows that
	\begin{equation*}
	\ker N_{i}^{\prime}\cong \mathrm{coker}(\xymatrix@=1.5pc{
		H_{1}(G,\bigoplus_{w\in S_{f}}H^{2}(E_{w},\Z_{2}(i)))\ar[r]&
		H_{1}(G,H^{0}(E,\Q_{2}/\Z_{2}(1-i))^{\vee})).}
	\end{equation*}
	Now we  give a description of $H_{1}(G_{v},H^{2}(E_{w},\Z_{2}(i)))$ and
	$H_{1}(G,H^{0}(E,\Q_{2}/\Z_{2}(1-i))^{\vee})$. We need some notation
	\begin{tabbing}
		\hspace{0.4cm} \= \hspace{0.4cm} \= \kill
		$E_{\infty}$ \> \>:  the cyclotomic
		$\Z_{2}$-extension of $E$.\\
		$G_{v}$\> \>:  the decomposition group of $v$ in $E/F$.\\
		$\Gamma_{v}$\> \>:  the decomposition group of $v$ in $F_{\infty}/F$.\\
		$H$\> \>: the $2$-part of the abelianization of $\mathrm{Gal}(E_{\infty}/F_{\infty})$.\\
		$H_{v}$\> \>:  the $2$-part of the abelianization of 	$\mathrm{Gal}(E_{w,\infty}/F_{v,\infty})$\\
		$L_{\infty}^+$\> \>: the maximal abelian $2$-extension of $F_{\infty}$, which is unramified at finite places \\
		\> \>   and completely decomposed at all primes above $2$.
	\end{tabbing}
	\begin{pro}\label{Proposition N surj}
		Let $i\geq 2$ be an integer and let $v$ be a finite place of $F$. If either $i$ is odd, or $i$ is even and $\sqrt{-1}\in F$, then we have
		\begin{tabbing}
			\hspace{0.4cm} \= \hspace{0.4cm} \= \kill
			$(1)$\>\>  $	H_{1}(G_{v},H^{2}(E_{w},\Z_{2}(i)))\cong
			H_{v}(i-1)_{\Gamma_{v}}$ and\\
			\> \> $	H_{1}(G,H^{0}(E,\Q_{2}/\Z_{2}(1-i))^{\vee})\cong H(i-1)_{\Gamma}.$\\
			$(2)$\> \>  $	\mathrm{coker}N_{i}\cong \mathrm{Gal}(L_{\infty}^+\cap
			E_{\infty}/F_{\infty})(i-1)_{\Gamma}$.
		\end{tabbing}
		In particular, the map $N_{i}$ is surjective if and only if $L_{\infty}^+\cap E_{\infty}=F_{\infty}$.
	\end{pro}
	\begin{proof} Using the assumption and (\ref{Tate lemma}),  we have
		\begin{equation*}
		H^{1}(\Gamma,\Q_{2}/\Z_{2}(1-i))=0\;\mbox{and}\;
		H^{1}(\Gamma_{v},\Q_{2}/\Z_{2}(1-i))=0.
		\end{equation*}
		Then  $(1)$ can be proved with the same argument of  Proposition $2.1$
		of \cite{Assim 12}. The second assertion is a direct consequence of the first one.
	\end{proof}
	To prove  a genus formula for the positive \'{e}tale wild kernel, we give a description of
	\begin{equation*}
	\ker
	N_{i}\cong \mathrm{coker}\widetilde{\alpha}.
	\end{equation*}
	Consider the
	following exact commutative diagram
	\begin{equation}\label{Diagram}
	\xymatrix@=1pc{      &
		H_{2}(G,\oplus_{w\in S_{f}}H^{2}(E_{w},\Z_{2}(i)))\ar[d]^-{\kappa}\\
		&
		H_{2}(G,H^{0}(E,\Q_{2}/\Z_{2}(1-i))^{\vee})\ar[d]\\
		H_{1}(G,H^{2}_{+}(G_{E,S},\Z_{2}(i)))\ar[r]^-{\widetilde{\alpha}}\ar[d]^-{\wr}&
		H_{1}(G,\widetilde{\oplus}_{w\in S_{f}}H^{2}(E_{w},\Z_{2}(i)))\ar[d]\\
		H_{1}(G,H^{2}_{+}(G_{E,S},\Z_{2}(i)))\ar[r]^-{\alpha}&
		H_{1}(G,\oplus_{w\in S_{f}}H^{2}(E_{w},\Z_{2}(i)))\ar[d]^-{\theta}\\
		&  H_{1}(G,H^{0}(E,\Q_{2}/\Z_{2}(1-i))^{\vee})}
	\end{equation}
	Then we have an exact sequence
	\begin{equation}\label{exact sequence of kappa alpha}
	\xymatrix@=1pc{0\ar[r]&\ker\widetilde{\alpha}\ar[r]&\ker\alpha\ar[r]&\mathrm{coker}\,\kappa\ar[r]&
		\mathrm{coker}\widetilde{\alpha}\ar[r]&\mathrm{coker}\alpha\ar[r]&
		\mathrm{Im}\theta\ar[r]&0.}
	\end{equation}
	\begin{deft}\label{Xi}
		We define the module $	X^{(i)}_{E/F}$ as
		\begin{equation*}
		X^{(i)}_{E/F}:=\mathrm{Im}(\xymatrix@=1.5pc{\mathrm{coker}\,\kappa\ar[r]&
			\mathrm{coker}\widetilde{\alpha}}),
		\end{equation*}
		where $\kappa$ is the homology map
		\begin{equation*}
		\xymatrix@=2pc {H_{2}(G,\oplus_{w\in S_{f}}H^{2}(E_{w},\Z_{2}(i)))\ar[r]&
			H_{2}(G,H^{0}(E,\Q_{2}/\Z_{2}(1-i))^{\vee})}.
		\end{equation*}
	\end{deft}
	So
	\begin{equation*}
	|X^{(i)}_{E/F}|\leq |H_{2}(G,H^{0}(E,\Q_{2}/\Z_{2}(1-i))^{\vee})|.
	\end{equation*}
	We have the  following comparison between $|(WK_{2i-2}^{\mbox{\'{e}t},+}E)_{G}|$ and $|WK_{2i-2}^{\mbox{\'{e}t},+}F|$.
	\begin{pro}\label{Pro genus formula}
		For any integer $i\geq 2$, we have
		\begin{equation*}
		\frac{|(WK_{2i-2}^{\mbox{\'{e}t},+}E)_{G}|}{|WK_{2i-2}^{\mbox{\'{e}t},+}F|}=\frac{|X^{(i)}_{E/F}|.|\mathrm{coker\alpha}|}
		{|H_{1}(G,H^{0}(E,\Q_{2}/\Z_{2}(1-i))^{\vee})|}
		\end{equation*}
	\end{pro}
	\begin{proof} On the one hand, by  $(\ref{exact sequence of kappa
			alpha})$ and the definition of $X^{(i)}_{E/F}$, we have an exact
		sequence
		\begin{equation*}
		\xymatrix@=1.5pc{0\ar[r]& X^{(i)}_{E/F}\ar[r]&
			\mathrm{coker}\widetilde{\alpha}\ar[r]&\mathrm{coker}\alpha\ar[r]&
			\mathrm{Im}\theta\ar[r]&0.}
		\end{equation*}
		On the other hand the exact commutative diagram $(\ref{exact
			commutaive diagram widetilde})$ shows that
		\begin{eqnarray*}
			\ker N^{\prime}_{i} &=& \mathrm{coker}(\xymatrix@=1.5pc{ H_{1}(G,\oplus_{w\in
					S_{f}}H^{2}(E_{w},\Z_{2}(i)))\ar[r]& H_{1}(G,H^{0}(G_{E,S},\Q_{2}/\Z_{2}(1-i))^{\vee})}  \\
			&=&\mathrm{coker}\theta.
		\end{eqnarray*}
		Since
		\begin{equation*}
		\mathrm{coker}N_{i}\cong \ker N_{i}^{\prime},\; \ker N_{i}\cong
		\mathrm{coker}\widetilde{\alpha}\;\;\mbox{and}\;\;
		|\mathrm{Im}\theta|.|\mathrm{coker}\theta|=|H_{1}(G,H^{0}(E,\Q_{2}/\Z_{2}(1-i))^{\vee})|,
		\end{equation*}
		we obtain
		\begin{equation*}
		\frac{|(WK_{2i-2}^{\mbox{\'{e}t},+}E)_{G}|}{|WK_{2i-2}^{\mbox{\'{e}t},+}F|}=\frac{|X^{(i)}_{E/F}|.|\mathrm{coker\alpha}|}
		{|H_{1}(G,H^{0}(E,\Q_{2}/\Z_{2}(1-i))^{\vee})|}.
		\end{equation*}
	\end{proof}
	Now we are going to compute
	$|\mathrm{coker\alpha}|$. For every $q\in \Z$,
	we have an isomorphism
	\begin{equation*}
	\widehat{H}^{q}(G,H^{2}_{+}(G_{E,S},\Z_{2}(i)))\cong
	\widehat{H}^{q+2}(G,H^{1}_{+}(E,\Z_{2}(i))),
	\end{equation*}
	given by cup-product (\cite[Proposition 3.1]{CKPS}),
	where $\widehat{H}^{\ast}(.,.)$ denotes the Tate cohomology. Then
	the commutative diagram
	\begin{equation*}
	\xymatrix@=2pc{
		H_{1}(G,H^{2}_{+}(G_{E,S},\Z_{2}(i)))\ar[d]^-{\wr}\ar[r]^-{\alpha}&
		H_{1}(G, \oplus_{w\in S_{f}}H^{2}(E_{w},\Z_{2}(i)))\ar[d]^-{\wr}\\
		\widehat{H}^{0}(G,H^{1}_{+}(E,\Z_{2}(i)))\ar[r]^-{\beta}& \widehat{H}^{0}(G,\oplus_{w\in
			S_{f}}H^{1}(E_{w},\Z_{2}(i)))}
	\end{equation*}
	shows that
	\begin{equation}\label{ coker alpha}
	\mathrm{coker}\alpha\cong  \mathrm{coker}\beta.
	\end{equation}
	\begin{deft}
		Let $H^{1,\mathcal{N}}_{+}(F,\Z_{2}(i))$ denote
		the kernel  of the map
		\begin{equation*}
		\xymatrix@=1.5pc{  H^{1}_{+}(F,\Z_{2}(i))\ar[r]&
			\oplus_{v\in S_{f}}
			\frac{H^{1}(F_{v},\Z_{2}(i))}{N_{G_{v}}H^{1}(E_{w},\Z_{2}(i))}}
		\end{equation*}
		where $N_{G_{v}}=\sum_{\sigma\in G_{v}}\sigma$ is the norm map.
	\end{deft}
	The isomorphism (\ref{ coker alpha})  shows that
	\begin{equation*}
	\mathrm{Im}(\alpha)\cong
	H^{1}_{+}(F,\Z_{2}(i))/H^{1,\mathcal{N}}_{+}(F,\Z_{2}(i)).
	\end{equation*}
	Hence
	\begin{equation*}
	|\mathrm{coker}\alpha|=\frac{\prod_{v\in
			S_{f}}|H_{1}(G_{v},H^{2}(E_{w},\Z_{2}(i)))|}{[H^{1}_{+}(F,\Z_{2}(i)):
		\mathcal{H}^{1,\mathcal{N}}_{+}(F,\Z_{2}(i))]}.
	\end{equation*}
	This yields our main result:
	\begin{theo}\label{theorem prin}
		Let $E/F$ be a Galois extension of number fields  with
		Galois group $G$. Then for every $i\geq 2$, we have
		\begin{equation*}
		\frac{|(WK_{2i-2}^{\mbox{\'{e}t},+}E)_{G}|}{|WK_{2i-2}^{\mbox{\'{e}t},+}F|}=\frac{|X^{(i)}_{E/F}|.\prod_{v\in
				S_{f}}|H_{1}(G_{v},H^{2}(E_{w},\Z_{2}(i)))|}
		{|H_{1}(G,H^{0}(E,\Q_{2}/\Z_{2}(1-i))^{\vee})|.[H^{1}_{+}(F,\Z_{2}(i)):
			\mathcal{H}^{1,\mathcal{N}}_{+}(F,\Z_{2}(i))]}.
		\end{equation*}
		\hfill $\square$
	\end{theo}
\begin{rem}
The above genus formula involves the order of the group $X^{(i)}_{E/F}$ which seems to be difficult to compute. If we make the following hypothesis $(\mathcal{H})$:\vskip 5pt
 "the morphism
		\begin{equation}\label{Hyp H}
 \kappa: \xymatrix@=1pc{	H_{2}(G,\oplus_{w\in S_{f}}H^{2}(E_{w},\Z_{2}(i)))\ar[r]& 	H_{2}(G,H^{0}(E,\Q_{2}/\Z_{2}(1-i))^{\vee})}
 \,\mbox{is surjective}",
	\end{equation}
then $X^{(i)}_{E/F}$ is trivial.\\
	This hypothesis  is satisfied if there is a prime $v_{0}$ of $F$ such that:
\begin{itemize}
	\item $H^{0}(E, \Q_{2}/\Z_{2}(1-i))\cong H^{0}(E_{w_{0}}, \Q_{2}/\Z_{2}(1-i))$ and
	\item $v_{0}$ is an  undecomposed $2$-adic prime, or $v_{0}$ is a totally and tamely ramified prime in $E/F$.
\end{itemize}
\end{rem}
	\begin{rem}
		The groups $H_{1}(G_{v},H^{2}(E_{w},\Z_{2}(i)))$ can be easily computed (at least if $i$ is odd, or $i$ is even and $\sqrt{-1}\in F$ by Proposition \ref{Proposition N surj}). The difficult part here is the norm index
		$[H^{1}_{+}(F,\Z_{2}(i)):
		\mathcal{H}^{1,\mathcal{N}}_{+}(F,\Z_{2}(i))]$. When $E/F$ is a relative quadratic extension, we obtain a  genus formula  involving the norm index
		$[D_{F}^{+(i)}:D_{F}^{+(i)}\cap N_{G}E^{\bullet}]$. Moreover, if $F$ has  one $2$-adic prime, we  use \cite[\S 4]{Assim mova} to give an explicit description of this norm index in terms of the ramification in $E/F$.
	\end{rem}
	\section{Relative quadratic extension case}
	In this section we focus on  relative quadratic extensions of number fields $E/F$  with Galois group $G$. For such extensions, we
	give a genus formula for the positive \'{e}tale wild kernel  involving the
	norm index $[D_{F}^{+(i)}:D_{F}^{+(i)}\cap N_{G}E^{\bullet}]$, where
	$D^{+(i)}_{F}$ is the  positive Tate kernel. \vskip 6pt
	First recall that  for every even  integer $i\geq 2$, Proposition \ref{Proposition DF and WF} says that the \'{e}tale wild kernel and the positive  \'{e}tale wild kernel coincide. A genus formula has been obtained
	by Kolster-Movahhedi \cite[Theorem 2.18]{Ko-Mo} for $i=2$, and by  Griffiths \cite[\S 4.3]{Griffiths}, as a generalization, for any even  integer $i\geq 2$.
	This  genus formula can be used to determine families of abelian $2$-extensions with trivial $2$-primary Hilbert  kernel  \cite{Ko-Mo 03,Lescop, Griffiths2}.
	\vskip 6pt
	Throughout this section we keep the notations of the previous sections and we assume that the integer $i\geq 2$ is odd.\vskip 6pt
	We need to calculate the order of $\mathrm{coker}\alpha $ (see Proposition \ref{Pro genus formula}). Since $G$ has order $2$,  we have the following exact commutative  diagram
	\begin{equation*}
	\xymatrix@=2pc{	\widehat{H}^{0}(G,H^{1}_{+}(E,\Z_{2}(i)))\ar[r]^-{\beta}\ar[d]^-{\wr}&
		\oplus_{v\in S_{f}}	\widehat{H}^{0}(G_{v},H^{1}(E_{w},\Z_{2}(i)))\ar[d]^-{\wr}\\
		H^{1}_{+}(F,\Z_{2}(i))/2/N_{G}(	H^{1}_{+}(E,\Z_{2}(i))/2)\ar[r]^-{\beta^{\prime}}&
		\oplus_{v\in S_{f}}	H^{1}(F_{v},\Z_{2}(i))/2/N_{G_{v}}(	H^{1}(E_{w},\Z_{2}(i))/2)}
	\end{equation*}	
	Likewise as in the global case, there exists a subgroup $D_{v}^{(i)}$ of $F_{v}^{\bullet}$
	containing $F_{v}^{\bullet^{2}}$ such that	
	\begin{equation*}
	H^{1}(F_{v},\Z_{2}(i))/2\cong D_{v}^{(i)}/F_{v}^{\bullet^{2}}
	\end{equation*}
	for each $v\in S_{f}$. Then, we have a natural isomorphism
	\begin{equation*}
	H^{1}(F_{v},\Z_{2}(i))/2/N_{G_{v}}(	H^{1}(E_{w},\Z_{2}(i))/2)\cong
	D_{v}^{(i)}/F_{v}^{\bullet^{2}}N_{G_{v}}(D_{w}^{(i)}),
	\end{equation*}
	where $w$ is a prime of $E$ above $v$. Hence,
	\begin{eqnarray*}
		\mathrm{coker\alpha}&\cong& 	\mathrm{coker\beta}\quad (\mbox{by (\ref{ coker alpha})}) \\
		&\cong& \mathrm{coker\beta^{\prime}}\\
		&\cong& \mathrm{coker}(\delta:\xymatrix@=1.5pc{	H^{1}_{+}(F,\Z_{2}(i))/2/N_{G}(	H^{1}_{+}(E,\Z_{2}(i))/2)\ar[r]& \oplus_{v\in S_{f}}D_{v}^{(i)}/F_{v}^{\bullet^{2}}N_{G_{v}}(D_{w}^{(i)}})).
	\end{eqnarray*}
	On the one hand, there exists a surjective map
	\begin{equation}\label{map surj}
	\xymatrix@=1.5pc{	H^{1}_{+}(F,\Z_{2}(i))/2/N_{G}(	H^{1}_{+}(E,\Z_{2}(i))/2)\ar@{->>}[r]&D_{F}^{+(i)}/F^{\bullet^{2}}N_{G}D_{E}^{+(i)}}
	\end{equation}
	Indeed, the exact sequence
	\begin{equation*}
	\xymatrix@=1.5pc{\oplus_{v\mid\infty}H^{0}(F_{v},\Z_{2}(i))\ar[r]& H^{1}_{+}(F,\Z_{2}(i))\ar[r]& H^{1}(F,\Z_{2}(i))\ar[r]& \oplus_{v\mid \infty}H^{1}(F_{v},\Z_{2}(i))}
	\end{equation*}
	induces the exact sequence
	\begin{equation*}
	\xymatrix@=1.5pc{ H^{1}_{+}(F,\Z_{2}(i))/2\ar[r]& H^{1}(F,\Z_{2}(i))/2\ar[r]& \oplus_{v\mid \infty}H^{1}(F_{v},\Z_{2}(i)).}
	\end{equation*}
	Then, by the definition of $D_{F}^{+(i)}$, we have a surjective
	map
	\begin{equation*}
	\xymatrix@=1.5pc{ H^{1}_{+}(F,\Z_{2}(i))/2\ar[r]&
		D_{F}^{+(i)}/F^{\bullet^{2}}\ar[r]&0.}
	\end{equation*}
	Hence, the surjectivity of the map (\ref{map surj}) follows from the exact commutative diagram
	\begin{equation*}
	\xymatrix@=2pc{ H^{1}_{+}(E,\Z_{2}(i))/2\ar[r]\ar[d]^-{N_{G}}&
		D_{E}^{+(i)}/E^{\bullet^{2}}\ar[r]\ar[d]^-{N_{G}}&0 \\
		H^{1}_{+}(F,\Z_{2}(i))/2\ar[r]\ar@{->>}[d]&
		D_{F}^{+(i)}/F^{\bullet^{2}}\ar[r]\ar@{->>}[d]&0\\
		H^{1}_{+}(F,\Z_{2}(i))/2/N_{G}(	H^{1}_{+}(E,\Z_{2}(i))/2)\ar[r]&D_{F}^{+(i)}/F^{\bullet^{2}}N_{G}D_{E}^{+(i)}\ar[r]&0}
	\end{equation*}	
	On the other hand, 	since $i$ is odd,
	the canonical sujection map
	\begin{equation*} \xymatrix@=2pc{\psi^{(i)}_{v}:\;D_{v}^{(i)}/F_{v}^{\bullet^{2}}N_{G_{v}}(D_{w}^{(i)})\ar@{->>}[r]& D_{v}^{(i)}/D_{v}^{(i)}\cap N_{G_{v}}(E_{w}^{\bullet})}
	\end{equation*}	
	is an isomorphism, as a consequence of \cite[Lemma 4.2.1]{Griffiths}. Therefore,
	\begin{equation*}
	\mathrm{Im}(\delta) \cong
	\mathrm{Im}(\xymatrix@=1.5pc{D_{F}^{+(i)}/F^{\bullet^{2}}N_{G}D_{E}^{+(i)}\ar[r]&	\oplus_{v\in S_{f}} D_{v}^{(i)}/D_{v}^{(i)}\cap N_{G_{v}}(E_{w}^{\bullet})}),
	\end{equation*}
	it follows that
	\begin{eqnarray*}
		|\mathrm{Im}(\delta)|&=& [D^{+(i)}_{F}:D^{+(i)}_{F}\cap \cap_{v\in S_{f}} N_{G_{v}}E_{w}^{\bullet} ]\\
		&=& [D^{+(i)}_{F}:D^{+(i)}_{F}\cap \cap_{v} N_{G_{v}}E_{w}^{\bullet} ]
	\end{eqnarray*}
	where, in the last equality, $v$ runs through all  places of $F$. Then, using the Hasse norm theorem ($G$ is cyclic), we get
	\begin{equation*}
	|\mathrm{Im}(\delta)|=[D^{+(i)}_{F}:D^{+(i)}_{F}\cap N_{G}E^{\bullet} ].
	\end{equation*}
	
	Therefore, we obtain
	\begin{equation*}
	|\mathrm{coker}\alpha|=\frac{\prod_{v\in
			S_{f}}|H_{1}(G_{v},H^{2}(E_{w},\Z_{2}(i)))|}{[D^{+(i)}_{F}:D^{+(i)}_{F}\cap N_{G}E^{\bullet} ]}.
	\end{equation*}
	Moreover, the order of the group $X^{(i)}_{E/F}$ (see Definition \ref{Xi}) is at most  $2$:
	\begin{equation*}
	|X^{(i)}_{E/F}|=2^{t},\; t\in\{0,1\}.
	\end{equation*}
	Indeed, by the definition of $X^{(i)}_{E/F}$, we have
	\begin{equation*}
	|X^{(i)}_{E/F}|\leq |H_{2}(G,H^{0}(E,\Q_{2}/\Z_{2}(1-i))^{\vee})|\leq 2,
	\end{equation*}
	since $G$ is cyclic of order $2$.
	Hence, by Proposition \ref{Pro genus formula}, we get
		\begin{equation*}
		\frac{|(WK_{2i-2}^{\mbox{\'{e}t},+}E)_{G}|}{|WK_{2i-2}^{\mbox{\'{e}t},+}F|}=\frac{2^{t}\cdot\prod_{v\in
				S_{f}}|H_{1}(G_{v},H^{2}(E_{w},\Z_{2}(i)))|}
		{|H_{1}(G,H^{0}(E,\Q_{2}/\Z_{2}(1-i))^{\vee})|\cdot[D^{+(i)}_{F}:D^{+(i)}_{F}\cap N_{G}E^{\bullet} ]}
		\end{equation*}
		where $t\in\{0,1\}$.
		\vskip 6pt	
	First, assume that  $E\subseteq F_{\infty}$. Then for every finite prime $v$, we have
	\begin{equation*}
	H_{1}(G_{v},H^{2}(E_{w},\Z_{2}(i)))=0,
	\end{equation*}
	by Proposition \ref{Proposition N surj}.
	Since $G$ is cyclic, we also have
	\begin{equation*}
	|H_{2}(G_{v},H^{2}(E_{w},\Z_{2}(i)))|=| H_{1}(G_{v},H^{2}(E_{w},\Z_{2}(i)))|=1,
	\end{equation*}
	and then, using the commutative diagram (\ref{Diagram}), we see that
	\begin{equation*}
	\mathrm{coker}\widetilde{ \alpha}=0.
	\end{equation*}
	Then the map
	\begin{equation*}
	\xymatrix@=2pc{N_{i}: (	WK_{2i-2}^{\mbox{\'{e}t},+}E)_{G}\ar[r]& WK_{2i-2}^{\mbox{\'{e}t},+}F}
	\end{equation*}
	is an isomorphism, as a consequence of Proposition \ref{Proposition N surj} and the fact that  $\ker N_{i}\cong  \mathrm{coker}\widetilde{ \alpha}$.	\vskip 6pt	
	Now, if  $E\nsubseteq F_{\infty}$, then
	\begin{equation*}
	|H_{1}(G,H^{0}(E,\Q_{2}/\Z_{2}(1-i))^{\vee})|=2
	\end{equation*}
	and if $v\nmid 2$ is a finite ramified prime in $E/F$ or $v$ is a $2$-adic prime such that  $E_{w}\cap F_{v,\infty}\neq E_{w}$, then
	\begin{equation*}
	|H_{1}(G_{v},H^{2}(E_{w},\Z_{2}(i)))|=2.
	\end{equation*}
	Since $G$ is cyclic, $|(WK_{2i-2}^{\mbox{\'{e}t},+}E)_{G}|=|(WK_{2i-2}^{\mbox{\'{e}t},+}E)^{G}|$. 	We can now formulate the genus formula for a relative quadratic extension:
	\begin{pro}
		Let $E/F$ be a relative quadratic extension of number
		fields with Galois group $G$ and  let $R_{E/F}$ be the set of  finite  primes tamely ramified in $E/F$ or $2$-adic primes such that $E_{w}\cap F_{v,\infty}\neq E_{w}$.
		Then for any odd positive integer $i\geq 2$
		\begin{enumerate}[label=(\roman*)]
			\item if $E\subseteq F_{\infty}$ then the positive \'{e}tale wild kernel satisfies Galois codescent, and
			\item if $E\nsubseteq F_{\infty}$,
			\begin{equation*} \frac{|(WK_{2i-2}^{\mbox{\'{e}t},+}E)^{G}|}{|WK_{2i-2}^{\mbox{\'{e}t},+}F|}=\frac{2^{r(E/F)-1+t}}{
				[D_{F}^{+(i)}:D_{F}^{+(i)}\cap N_{G}E^{\bullet}]}
			\end{equation*}
			where $r(E/F)=|R_{E/F}|$ and $t\in\{0,1\}$.	
		\end{enumerate}
	\hfill$\square$
	\end{pro}
Recall that under the hypothesis $(\mathcal{H})$  the group $X_{E/F}^{(i)}$ is trivial. Moreover, if $E/F$ is a relative quadratic extension of number
		fields, the hypothesis $(\mathcal{H})$ is satisfied precisely when the set $R_{E/F}$ is nonempty. We obtain
	\begin{coro}\label{Corollary quadratic}
		Let $E/F$ be a relative quadratic extension of number fields such that
		$E\nsubseteq F_{\infty}$ with Galois group $G$. If $R_{E/F}\neq \emptyset$ then we have
		\begin{equation*}	\frac{|(WK_{2i-2}^{\mbox{\'{e}t},+}E)^{G}|}{|WK_{2i-2}^{\mbox{\'{e}t},+}F|}=\frac{2^{r(E/F)-1}}{
			[D_{F}^{+(i)}:D_{F}^{+(i)}\cap N_{G}E^{\bullet}]}.
		\end{equation*}
	\end{coro}
	\begin{exe} Assume that $F$ has one $2$-adic prime and that the set $R_{E/F}\neq\emptyset$. Let $S=S_{2}\cup R_{E/F}$. Since $E/F$ is quadratic (hence cyclic), we have
		\begin{equation*}
			[D_{F}^{+(i)}:D_{F}^{+(i)}\cap N_{G}E^{\bullet}]=[D_{F}^{+(i)}:D_{F}^{+(i)}\cap\cap_{v\in S-S_{2}}N_{E_{w}/F_{v}}(E_{w}^{\bullet})]
		\end{equation*}
	by Hasse's local-global norm principle.\\
	Recall that a set $T$ of primes containing $S_{2}$ is $D_{F}^{+(i)}$-primitive for $(F,2)$ if the canonical map
	\begin{equation*}
	\xymatrix@=2pc{D_{F}^{+(i)}/F^{\bullet^{2}}\ar[r]& \oplus_{v\in T\backslash S_{2}}D^{(i)}_{v}/F_{v}^{\bullet^{2}}}
	\end{equation*}
is surjective \cite[Definition 4.4]{Assim mova}. Let $T$ be a maximal 	$D_{F}^{+(i)}$-primitive set for $(F,2)$ contained in $S$ and $t_{i}^{+}=|T\backslash S_{2}|$. Then using \cite[ Proposition 4.11]{Assim mova}, we have
	\begin{equation*}
		[D_{F}^{+(i)}:D_{F}^{+(i)}\cap N_{G}E^{\bullet}]=	[D_{F}^{+(i)}:D_{F}^{+(i)}\cap\cap_{v\in S-S_{2}}N_{E_{w}/F_{v}}(E_{w}^{\bullet})]= 2^{t^{+}_{i}}.
	\end{equation*}
	In fact, $t^{+}_{i}=\dim_{\mathbf{F}_2} \mathrm{Im}(\xymatrix@=1.5pc{D_{F}^{+(i)}/F^{\bullet^{2}}\ar[r]& \oplus_{v\in S\backslash S_{2}}D^{(i)}_{v}/F_{v}^{\bullet^{2}}})$.\\
		Now, let $E$ be a quadratic field such that $E$ is not contained  in the cyclotomic $\Z_{2}$-extension of $\Q$. Since the narrow class group of $\Q$ is trivial, the set $R_{E/\Q}\neq\emptyset$. Moreover,  $WK_{2i-2}^{\mbox{\'{e}t},+}\Q=0$,  by Proposition \ref{Prop ind}.
	By  Corollary \ref{Corollary quadratic} it follows that
		\begin{equation*}
		|(WK_{2i-2}^{\mbox{\'{e}t},+}E)^{G}|=2^{r(E/\Q)-1-t_{i}^{+}}
		\end{equation*}
		for all odd integer $i\geq 2$. 

Let  $rk_{2}(WK_{2i-2}^{\mbox{\'{e}t},+}E)$ denote the $2$-rank of the \'{e}tale positive wild kernel $WK_{2i-2}^{\mbox{\'{e}t},+}E$. The extension  $E/\Q$ being quadratic, we have
			\begin{equation*}
		|(WK_{2i-2}^{\mbox{\'{e}t},+}E)^{G}|=2^{rk_{2}(WK_{2i-2}^{\mbox{\'{e}t},+}E)}
		\end{equation*}
		(see e.g. \cite[Proposition 1.3]{Griffiths2}). Hence
			\begin{equation*}
		rk_{2}(WK_{2i-2}^{\mbox{\'{e}t},+}E)=r(E/\Q)-1-t_{i}^{+}.
		\end{equation*}	
		Since $\dim_{\F_{2}}D_{\Q}^{+(i)}/\Q^{\bullet^{2}}=1$ \cite[\S 6]{Assim mova},  $t_{i}^{+}\in\{0,1\}$. Therefore,
		\begin{equation*}
			rk_{2}(WK_{2i-2}^{\mbox{\'{e}t},+}E)=\begin{cases}
			r(E/\Q)-2,\, \mbox{if}\; T\backslash S_{2}\neq \emptyset\\
			          \,\\
				r(E/\Q)-1,\,\mbox{if}\;T\backslash S_{2}= \emptyset
			\end{cases}
		\end{equation*}	
	In particular, $ WK_{2i-2}^{\mbox{\'{e}t},+}E$ vanishes if and only if  $E$ is unramified  outside a set  $\{ 2,\infty, \ell\}$ with $\ell\equiv\pm 3\mod 8$.
	\end{exe}
	
\end{document}